\documentclass[11pt]{amsart}
\usepackage{amssymb,amsmath,amsthm,amsfonts,amsopn,url,bbm, hyperref}
\usepackage{enumitem}
\theoremstyle{plain}
\newtheorem{thm}{Theorem}[section]

\newtheorem{prop}[thm]{Proposition}
\newtheorem{lemma}[thm]{Lemma}
\newtheorem{cor}[thm]{Corollary}

\theoremstyle{definition}
\newtheorem{defn}[thm]{Definition}
\newtheorem*{defn*}{Definition}
\newtheorem{example}[thm]{Example}
\newtheorem*{example*}{Example}

\theoremstyle{remark}
\newtheorem{rmk}[thm]{Remark}
\newtheorem*{rmk*}{Remark}
\newcommand{\field}[1]{\mathbbm{#1}}

\newcommand{\N}{\field{N}}

\newcommand{\ideal}[1]{\mathfrak{#1}}

\newcommand{\p}{\ideal{p}}
\newcommand{\q}{\ideal{q}}
\newcommand{\ia}{\ideal{a}}

\newcommand{\cal}{\mathcal}

\newcommand{\cF}{\cal{F}}
\newcommand{\cf}{\cal{F}_\rf}

\newcommand{\cI}{\cal{I}}

\newcommand{\ra}{\rightarrow}

\newcommand{\rb}{{\rm b}}
\newcommand{\rc}{{\rm c}}

\newcommand{\rd}{{\rm d}}
\newcommand{\rf}{{\rm f}}
\newcommand{\rg}{{\rm g}}

\newcommand{\rs}{{\rm s}}

\DeclareMathOperator{\id}{\cI}
\DeclareMathOperator{\idf}{\cI_\rf}
\DeclareMathOperator{\ffid}{\cf}
\DeclareMathOperator{\fid}{\cF}
\DeclareMathOperator{\fsub}{\overline{\fid}}
\DeclareMathOperator{\nzd}{nzd}
\newcommand{\sig}{{\sigma}}
\newcommand{\sigf}{{\sigma_\rf}}
\newcommand{\rad}{\mathrm{rad}}

\author{Neil Epstein}
\address{Department of Mathematical Sciences \\ George Mason University \\ Fairfax, VA  22030}
\email{nepstei2@gmu.edu}

\title{Semistar operations and standard closure operations}

\date{August 28, 2013}
\begin{document}
\begin{abstract}
Let $R$ be a commutative ring.  It is shown that there is an order isomorphism between a popular class of finite type closure operations on the ideals of $R$ and the poset of semistar operations of finite type.
\end{abstract}
\dedicatory{Dedicated to Professor Marco Fontana on the occasion of his $65$th birthday.}

\maketitle

\section{Introduction}
Star operations have long represented an important topic in commutative ring theory, especially among non-Noetherian practitioners, ever since they were introduced by Krull \cite{Kr-idealbook} as a generalization of some elements of number theory.  Later \cite{OkMa-semistar}, the definition was extended to the notion of ``semistar operations'', a topic that also has attracted much interest in the non-Noetherian world.

On the other hand and in parallel, particular closure operations on ideals (e.g. radical, integral closure \cite{HuSw-book}, and tight closure \cite{HuTC}) have been investigated intensively (and to great effect) in Noetherian rings.  Axiomatization of these ideas occurred even more recently, as part of an attempt toward solving some conjectures in mixed characteristic.  See the recent survey \cite{nme-guide2}.

It is clear how to relate \emph{star} operations to certain kinds of closure operations \cite[Proposition 4.1.3]{nme-guide2}.  On the other hand, \emph{semistar} operations appear to be further afield at first glance, since the star of an ideal is not always an ideal.  In this article, perhaps counterintuitively, we show that if anything, the notion of the semistar operation is \emph{more} well-suited to building a bridge between the Noetherian and non-Noetherian worlds than the notion of the star operation.

Namely, we define a particular kind of closure operation (called a \emph{standard} closure operation), a rubric under which most of the closure operations in common use fit (with the radical being the primary counterexample). Then in the main theorem (Theorem~\ref{thm:main}), we show that there is a \emph{one-to-one, order-preserving correspondence} between the standard finite-type closure operations on a ring and the finite-type semistar operations on a ring.  In doing so, we launch a basic investigation on various operations on the poset of closure operations on ideals, namely the \emph{finitization} and the \emph{standardization} of a closure operation.

In the last section of the paper, we investigate the smallest standard closure operation above the radical.  We characterize it in terms of the total ring of quotients (Theorem~\ref{thm:rads}) and give an algorithm for its computation for any ideal that has a primary decomposition (Theorem~\ref{thm:radspdec}).

It is hoped that that this work will foster greater cooperation between non-Noetherian and Noetherian researchers on either side of this hitherto-unseen bridge.

\section{Closure operations and semistar operations}
Throughout, $R$ is a commutative ring with unity, and $Q$ is its total ring of fractions.  (However, the reader loses nothing by assuming that $R$ is an integral domain.)  Let $\id(R)$ (resp. $\idf(R)$) denote the set of ideals (resp. finitely generated ideals) of $R$.  We denote by $\ffid(R)$ (resp. $\fsub(R)$) the set of finitely generated $R$-submodules (resp. arbitrary $R$-submodules) of $Q$.  The set of non-zerodivisors of $R$ will be denoted $\nzd(R)$.  The set of fractional ideals of $R$ (i.e. those elements $A \in \fsub(R)$ such that $xA \subseteq R$ for some $x\in \nzd(R)$) is denoted  $\fid(R)$.  

\noindent \textbf{Convention:} Operations on $\id(R)$ and $\idf(R)$ will be \emph{superscripted}, whereas operations on $\fid(R)$, $\ffid(R)$, and $\fsub(R)$ will be \emph{subscripted}.

\begin{defn}
A set map $\rc: \id(R) \ra \id(R)$ is called an (ideal) \emph{preclosure operation} if it satisfies the following conditions for all $I, J \in \id(R)$: \begin{itemize}
\item (Extension) $I \subseteq I^\rc$, and
\item (Order-preservation) If $J \subseteq I$, then $J^\rc \subseteq I^\rc$.
\end{itemize}
A preclosure operation $\rc$ is called a \emph{closure operation} if it also satisfies
\begin{itemize}
\item (Idempotence) $(I^\rc)^\rc = I^\rc$.
\end{itemize}
\end{defn}

For more information on closure operations, the reader may consult the recent survey article \cite{nme-guide2}.

\begin{defn}
Let $\rc: \id(R) \ra \id(R)$ be a preclosure operation. We say that $\rc$ is: \begin{itemize}
\item \emph{of finite type} if $I^\rc = \bigcup\{J^\rc \mid J \subseteq I \text{ and } J \in \idf(R)\}$,
\item \emph{weakly prime}\footnote{This name is chosen since it is a weakening of the notion of semi-prime from \cite{Pet-asym}.} if for all $w\in \nzd(R)$ and $I \in \id(R)$, $w \cdot I^\rc \subseteq (wI)^\rc$.
\item \emph{standard} if for all $w\in \nzd(R)$ and $I \in \id(R)$, $((wI)^\rc :_R w) = I^\rc$.
\end{itemize}
\end{defn}

\begin{lemma}
Any standard preclosure operation is weakly prime.
\end{lemma}

\begin{proof}Elementary. \end{proof}

\begin{defn}
Let $\rc$ be a preclosure operation.  Define $\rc_\rf: \id(R) \ra \id(R)$ (the \emph{finitization} of $\rc$) as follows: \[
I^{\rc_\rf} := \bigcup\{J^\rc  \mid J \subseteq I \text{ and } J \in \idf(R)\}.
\]
Define $\rc_\rs: \id(R) \ra \id(R)$ (the \emph{standardization} of $\rc$) as follows: \[
I^{\rc_\rs} := \bigcup\{ ((wI)^\rc :_R w) \mid w \in \nzd(R)\}.
\]
\end{defn}

It is not \emph{a priori} clear from the definitions, given a (pre)closure operation $\rc$, that $\rc_\rf$ or $\rc_\rs$ is even a (pre)closure operation, much less one with desired properties.  However, we have the following

\begin{prop}
Let $\rc: \id(R) \ra \id(R)$ be a preclosure operation.
\begin{enumerate}
\item\label{it:cf} $\rc_\rf$ is a preclosure operation of finite type, weakly prime if $\rc$ is, standard if $\rc$ is, a closure operation if $\rc$ is, and $\rc_\rf \leq \rc$.  In fact, $\rc_\rf$ is the \emph{largest} finite type preclosure operation $\rd$ such that $\rd \leq \rc$.
\item\label{it:cs} If $\rc$ is weakly prime, then $\rc_\rs$ is a standard preclosure operation, of finite type if $\rc$ is, and $\rc \leq \rc_\rs$.  In fact, $\rc_\rs$ is the \emph{smallest} standard preclosure operation $\rd$ such that $\rc \leq \rd$.
\item\label{it:cs2} If $\rc$ is a weakly prime \emph{closure} operation of finite type, then $\rc_\rs$ is a standard closure operation of finite type.
\end{enumerate} 
\end{prop}

\begin{proof}
Part (\ref{it:cf}):
Let $I \in \id(R)$.  To see that $I^{\rc_\rf}$ is an ideal, let $x, y \in I^{\rc_\rf}$ and $r\in R$.  Then there exist finitely generated ideals $J, K \subseteq I$ such that $x\in J^\rc$ and $y \in K^\rc$.  It follows that $rx \in J^\rc \subseteq I^{\rc_\rf}$, and that $x+y \in (J+K)^\rc \subseteq I^{\rc_\rf}$, by order preservation and since $J+K$ is a finitely generated subideal of $I$.

Order-preservation is clear.

To see that $\rc_\rf$ satisfies extension, note that $I \subseteq \bigcup \{(x)^\rc \mid x \in I\} \subseteq I^{\rc_\rf}$.

$\rc_\rf$ is of finite type by definition.

Now suppose $\rc$ is weakly prime.  Let $w \in \nzd(R)$, $I$ an ideal, and $x \in I^{\rc_\rf}$.  Then there is some finitely generated ideal $K \subseteq I$ such that $x\in K^\rc$.  Then since $\rc$ is weakly prime, $wx \in (wK)^\rc \subseteq (wI)^{\rc_\rf}$, finishing the proof that $\rc_\rf$ is weakly prime.

Suppose $\rc$ is standard.  Let $w\in \nzd(R)$, $I$ an ideal, and $x\in R$ such that $wx \in (wI)^{\rc_\rf}$.  Then there is a finitely generated ideal $J \subseteq wI$ such that $wx \in J^\rc$.  But then there exist $r_1, \ldots, r_n \in R$ such that $J = (w r_1, \ldots, wr_n)$.  In particular, each $w r_i \in J \subseteq wI$, so since $w$ is regular, it follows that $r_i \in I$, so that $K = (r_1, \ldots, r_n) \subseteq I$ and $J = wK$.  Thus, $x \in ((wK)^\rc : w) = K^\rc$ since $\rc$ is standard.  Then by definition, $x \in I^{\rc_\rf}$, so that $\rc_\rf$ is standard.

If $\rc$ is a closure operation, we show here that $\rc_\rf$ is idempotent.  Let $r \in (I^{\rc_\rf})^{\rc_\rf}$.  Then there is a finitely generated ideal $J \subseteq I^{\rc_\rf}$ with $r \in J^\rc$.  Say $J = (r_1, \dotsc, r_n)$.  Since each $r_i \in I^{\rc_\rf}$, there exist finitely generated ideals $J_i$ with $J_i \subseteq I$ and $r_i \in (J_i)^\rc$.  Let $K := \sum_{i=1}^n J_i$.  Then $J \subseteq K^\rc$, $K$ is finitely generated, $K \subseteq I$, and $r \in J^\rc \subseteq (K^\rc)^\rc = K^\rc$.  Thus, $r \in I^{\rc_\rf}$.

As for maximality of $\rc_\rf$, let $\rd$ be a finite-type preclosure operation with $\rd \leq \rc$.  Let $I$ be an ideal and $a \in I^\rd$.  Since $\rd$ is of finite type, there is some finitely generated ideal $J \subseteq I$ with $a\in J^\rd$.  Since $\rd \leq \rc$, we have $a \in J^\rc$.  Thus, $a \in I^{\rc_\rf}$, so $\rd \leq \rc_\rf$, as required.

Part (\ref{it:cs}): Let $I \in \id(R)$. To see that $I^{\rc_\rs}$ is an ideal, let $x, y \in I^{\rc_\rs}$ and $r\in R$.  Then there exist $v,w \in \nzd(R)$ such that $x \in ((v I)^\rc :v)$ and $y \in ((w I)^\rc :w)$.  Clearly $rx \in ((vI)^\rc :v) \subseteq I^{\rc_\rs}$.  Moreover, we have \[
vw(x+y) = w(vx) + v(wy) \subseteq (w \cdot (vI)^\rc) + (v \cdot (wI)^\rc) \subseteq (vw \cdot I)^\rc,
\]
with the last containment since $\rc$ is weakly prime.  Hence, $x+y \in ((vw  I)^\rc : vw) \subseteq I^{\rc_\rs}$.

Extension follows since $I \subseteq I^\rc = ((1I)^\rc :1) \subseteq I^{\rc_\rs}$.  Order-preservation follows since if $J \subseteq I$, we have $((wJ)^\rc :w) \subseteq ((wI)^\rc :w)$ for any $w$.

To see that $\rc_\rs$ is standard, let $I$ be an ideal and $w \in \nzd(R)$.  If $x \in I^{\rc_\rs}$, then $vx \in (vI)^\rc$ for some $v\in \nzd(R)$, whence $vwx \in w (vI)^\rc \subseteq (v\cdot wI)^\rc$, whence $wx \in ((v wI)^\rc :_R v) \subseteq (wI)^{\rc_\rs}$, so that $x \in ((wI)^{\rc_\rs} :w)$.  Conversely, if $x \in ((wI)^{\rc_\rs} :w)$, then $wx \in (wI)^{\rc_\rs}$, which means that for some $v \in \nzd(R)$, we have $vwx \in (vwI)^\rc$, so that $x \in ((vwI)^\rc :vw) \subseteq I^{\rc_\rs}$, as required.

Now, suppose $\rc$ is of finite type, and let $x \in I^{\rc_\rs}$.  Then there is some $w\in \nzd(R)$ with $wx \in (wI)^\rc$.  Since $\rc$ is of finite type, there is an ideal $J \subseteq wI$ with $J \in \idf(R)$ and $wx \in J^\rc$.  It follows that $J = wK$ for some $K \in \idf(R)$ with $wx \in (wK)^\rc$, and $K \subseteq I$ since $w$ is $R$-regular.  Thus $x \in K^{\rc_\rs}$, so that $\rc_\rs$ is of finite type.

It is obvious that $\rc \leq \rc_\rs$.  As for minimality of $\rc_\rs$, let $\rd$ be a standard preclosure operation with $\rc \leq \rd$.  Let $I$ be an ideal and take any $a \in I^{\rc_\rs}$.  Then for some $w \in \nzd(R)$, we have $wa \in (wI)^\rc \subseteq (wI)^\rd$ (since $\rc \leq \rd$).  That is, $a \in ((wI)^\rd :_R w) \subseteq I^\rd$ since $\rd$ is standard, so that $\rc_\rs \leq \rd$ as required.

Part (\ref{it:cs2}): Assuming $\rc$ is a weakly prime closure operation of finite type, we need only show that $\rc_\rs$ is idempotent.  Let $x \in (I^{\rc_\rs})^{\rc_\rs}$.  Then $wx \in (w (I^{\rc_\rs}))^\rc$ for some $w \in \nzd(R)$.  Since $\rc$ is of finite type, there is some finitely generated ideal $J$ with $J \subseteq I^{\rc_\rs}$ and $wx \in (wJ)^\rc$.  Say $J = (r_1, \dotsc, r_n)$.  Then for $i=1, \dotsc, n$, there exist $v_1, \dotsc, v_n \in \nzd(R)$ such that $v_i r_i \in (v_i I)^\rc$.  Let $v = \prod_{i=1}^n v_i$.  It follows from weak primeness of $\rc$ that $v r_i \in (v I)^\rc$ for $i=1, \dotsc, n$, and hence that $vJ \subseteq (vI)^\rc$.  Thus, \[
vwx \in v(wJ)^\rc \subseteq (wvJ)^\rc \subseteq (w(vI)^\rc)^\rc \subseteq ((wvI)^\rc)^\rc = (vwI)^\rc,
\]
since $\rc$ is both weakly prime and idempotent.  It follows that $x \in ((vwI)^\rc :_R vw) \subseteq I^{\rc_\rs}$, as required.
\end{proof}

Those familiar with the usual closure operations will recognize that Frobenius, plus, and integral closure are standard closure operations of finite type.  Tight closure \cite{HHmain} is also a standard preclosure operation, and in a Noetherian ring, it is a closure operation \cite[Proposition 4.1]{HHmain}.  In general, though, it is neither of finite type nor idempotent \cite{nme-nntc}.  In a Noetherian local ring $R$ of prime characteristic where $\ia$ is a fixed proper ideal, $\ia$-tight closure is another example of a preclosure operation  that is typically not idempotent \cite[Remark 1.4(1)]{HaYo-atc} but is easily seen to be standard.  Radical is a finite type weakly prime closure operation that is almost never standard (see Corollary~\ref{cor:radns}).

\begin{defn}\cite{OkMa-semistar}
A set map $\star: \fsub(R) \ra \fsub(R)$ is a \emph{semistar operation} provided it satisfies the following properties for all $A, B \in \fsub(R)$ and all units $u$ of $Q$: \begin{itemize}
\item (Extension) $A \subseteq A_\star$.
\item (Order-preservation) If $A \subseteq B$ then $A_\star \subseteq B_\star$.
\item (Idempotence) $(A_\star)_\star = A_\star$.
\item (Divisibility) $u \cdot A_\star = (u A)_\star$.
\end{itemize}
We say that $\star$ is \emph{of finite type} if $A_\star = \bigcup\{ B_\star \mid B \subseteq A \text{ and } B \in \ffid(R) \}$.
\end{defn}

\begin{rmk}
The above definition is slightly different from the one in the literature, in two respects.  First, we define it above for general commutative rings, whereas it was previously only explored for integral domains.  Secondly, we define semistar operations as operations on the set of \emph{all} $R$-submodules of $Q$, rather than restricting (as is traditionally done) to the nonzero $R$-submodules of $Q$.  To fit the previous literature into the system presented here, simply set $0_\star := 0$ for any previously defined semistar operation $\star$.
\end{rmk}

There is no canonical reference for semistar operations, though I might recommend \cite{FoLo-Kronecker}.  For star operations, see \cite[ch. 32--34]{Gil-MIT}.

Recall that $\star_\rf$ is defined similarly to $\rc_\rf$.  See \cite[p. 4]{OkMa-semistar}.

Examples include the trivial semistar operation, the b-operation, and certain canonical extensions of star operations.

\section{An order-isomorphism}
\begin{thm}[Main Theorem]\label{thm:main}
There is a bijection, order-preserving in both directions, between the poset of finite-type standard closure operations on $R$ and the poset of finite-type semistar operations on $R$.
\end{thm}

Before starting the proof, we give the constructions for both directions of the correspondence.

\begin{defn}
Let $\star$ be a semistar operation.  Then we define $\kappa(\star): \id(R) \ra \id(R)$ as follows: \[
I^{\kappa(\star)} := I_\star \cap R.
\]
\end{defn}

\begin{lemma}
If $\star$ is a semistar operation, then $\kappa(\star)$ is a standard closure operation.  Moreover, $\kappa(\star)$ is of finite type whenever $\star$ is.
\end{lemma}

\begin{proof}
Extension is clear, as is order-preservation.  For idempotence, let $r \in (I^{\kappa(\star)})^{\kappa(\star)}$.  Then $r \in (I_\star)_\star = I_\star$, but since we also have $r \in R$, it follows that $r \in I_\star \cap R = I^{\kappa(\star)}$.

To see standardness, let $w\in \nzd(R)$.  Then \[
w \cdot I^{\kappa(\star)} \subseteq (w \cdot I_\star ) \cap R = (w I)_\star \cap R = (w I)^{\kappa(\star)},
\]
where the first equality follows from divisibility of $\star$.  Hence $I^{\kappa(\star)} \subseteq ((w I)^{\kappa(\star)} : w)$.

For the reverse containment, let $r \in ((w I)^{\kappa(\star)} :_R w)$, where $w\in \nzd(R)$.  Then \[
wr \in (w I)^{\kappa(\star)} = (wI)_\star \cap R \subseteq (wI)_\star = w \cdot I_\star,
\]
where the last equality follows from divisibility of $\star$.  That is, there is some $f \in I_\star$ such that $wr = wf$.  But $w$ is a unit of $Q$, so $r = f \in I_\star \cap R = I^{\kappa(\star)}$.

Finally, assume that $\star$ is of finite type. Let $I$ be an ideal of $R$ and $r \in I^{\kappa(\star)} = I_\star \cap R$.  Then there is a finitely generated $R$-submodule $J$ of $I$ with $r \in J_\star$.  Hence, $r \in J_\star \cap R = J^{\kappa(\star)}$.  Since $J \in \idf(R)$, it follows that $\kappa(\star)$ is of finite type.
\end{proof}

\begin{defn}\label{def:sigf}
Let $\rc$ be a standard preclosure operation.  We define $\pi(\rc): \fid(R) \ra \fsub(R)$ as follows:

For an element $f \in Q$ and a fractional ideal $A$, pick a representation $f=r/z$, $r\in R$, $z \in \nzd(R)$ and choose $x \in \nzd(R)$ such that $xA \subseteq R$ (so that $xA \in \id(R)$).  Then we say that $f\in A_{\pi(\rc)}$ if $xr \in (z \cdot xA)^\rc$.

Next, we define $\sigf(\rc): \fsub(R) \ra \fsub(R)$ as follows: \[
A_{\sigf(\rc)} := \bigcup \{ B_{\pi(\rc)} \mid B \subseteq A \text{ and } B \in \ffid(R)\}.
\]
\end{defn}

\begin{prop}
Let $\rc$ be a standard preclosure operation. \begin{enumerate}
\item\label{it:pidef} $\pi(\rc)$ is well-defined on $\fid(R)$, independent of the choices of $x$ and $z$ above.
\item\label{it:sigmass} If $\rc$ is moreover a closure operation, then $\sigf(\rc)$ is a semistar operation of finite type.
\end{enumerate} 
\end{prop}

\begin{proof}
Part (\ref{it:pidef}): Given $A \in \fid(R)$ and $f\in Q$.  Let $x, z, r$ be as above, so that $f = r/z$ and $I := xA \subseteq R$, and suppose $f \in A_{\pi(\rc)}$ according to these choices -- that is, $xr \in (zI)^\rc$.  Let $x', z', r'$ be another valid choice.  That is, $f = r'/z'$, $J := x'A \subseteq R$, and $x', z' \in \nzd(R)$.  We need to show that $f \in  A_{\pi(\rc)}$ according to the \emph{latter} set of choices as well.

Note that $xJ = xx'A = x'I$ and that $zr' = zz'f = z'r$.  Therefore, \begin{align*}
xz \cdot x'r' &= xx'zr' = xx'z'r = x'z'xr \\
&\in x'z'(zI)^\rc \subseteq (z'zx'I)^\rc = (z'z xJ)^\rc = (xz \cdot z'J)^\rc,
\end{align*}
so that by standardness of $\rc$, we have $x'r' \in (z'J)^\rc$, as required.

Part (\ref{it:sigmass}): Extension follows from the fact that $f \in (Rf)_{\pi(\rc)}$ for any $f \in Q$.  Order-preservation is clear.

For divisibility, let $A \in \fsub(R)$ and $f \in A_{\sigf(\rc)}$.  Then there is some $B \in \ffid(R)$ with $B \subseteq A$ and $f \in B_{\pi(\rc)}$.  Choose $x, z \in \nzd(R)$ and $r\in R$ such that $I := xB \subseteq R$ and $f = r/z$.  Let $u = w/y$ be a unit of $Q$ (so that $w, y$ are non-zerodivisors of $R$).  Then $xr \in (zI)^\rc$.  Multiplying both sides by $yw$, we get \[
yx \cdot wr = ywxr \in yw(zI)^\rc \subseteq (ywzI)^\rc = (yz \cdot wI)^\rc,
\]
which means precisely that $uf = \frac{wr}{yz} \in \left(\frac{wI}{yx}\right)_{\pi(\rc)} = (uB)_{\pi(\rc)}$.  But $uB$ is a finitely generated $R$-submodule of $uA$, whence $uf \in (uA)_{\sigf(\rc)}$.  Hence, $u \cdot A_{\sigf(\rc)} \subseteq (uA)_{\sigf(\rc)}$.

For the opposite direction, replace $u$ in the previous argument with $t := u^{-1}$ and $A$ with $D := uA$.  Then we have \[
u^{-1} \cdot (uA)_{\sigf(\rc)} = t \cdot D_{\sigf(\rc)} \subseteq (tD)_{\sigf(\rc)} = (u^{-1}uA)_{\sigf(\rc)} = A_{\sigf(\rc)}.
\]
Multiplying both sides by $u$, it follows that $(uA)_{\sigf(\rc)} \subseteq u \cdot A_{\sigf(\rc)}$.

By definition/construction, $\sigf(\rc)$ has finite type.

Finally, we must show idempotence.  Let $f=r/z \in (A_{\sigf(\rc)})_{\sigf(\rc)}$.  Then there exists $B \in \ffid(R)$ with $B \subseteq A_{\sigf(\rc)}$ and $f \in B_{\pi(\rc)}$.  Say $B = \sum_{i=1}^n R f_i$.  Since each $f_i \in A_{\sigf(\rc)}$, there exist finitely generated $R$-submodules $B_i$ of $A$ such that $f_i \in (B_i)_{\pi(\rc)}$.  Let $D := \sum_{i=1}^n B_i$.    There exists some $x\in \nzd(R)$ such that $xB \subseteq R$ and $xD \subseteq R$.  Let $r_i := x f_i$ for each $i$ (hence $r_i \in R$).  We have \[
xB = (r_1, \dotsc, r_n) \subseteq \sum_i (x B_i)^\rc \subseteq \left(\sum_i (x B_i)\right)^\rc = (xD)^\rc.
\]
Hence, since $f \in B_{\pi(\rc)}$, \[
xr \in (zxB)^\rc \subseteq (z \cdot [(xD)^\rc])^\rc \subseteq ((zxD)^\rc)^\rc = (zxD)^\rc, 
\]
which means that $f=r/z \in D_{\pi(\rc)}$.  But since $D$ is a finitely generated $R$-submodule of $A$, it follows that $f \in A_{\sigf(\rc)}$.
\end{proof}

\begin{thm}\label{thm:inv1}
For any semistar operation $\star$ on $R$, one has $\sigf(\kappa(\star)) = \star_\rf$.  In particular, if $\star$ is of finite type, then $\sigf(\kappa(\star)) = \star$.
\end{thm}

\begin{proof}
Let $A \in \fsub(R)$ and let $\frac{r}{z} = g \in A_{\star_\rf}$.  By definition of $\star_\rf$, there exists $B \in \ffid(R)$ with $B \subseteq A$ and $g \in B_\star$.  Then $r \in z \cdot B_\star = (zB)_\star$, so letting $x \in \nzd(R)$ be such that $xB \subseteq R$, we have \[
xr \in x\cdot ((zB)_\star \cap R) \subseteq (x(zB)_\star) \cap R = (zxB)_\star \cap R = (z \cdot xB)^{\kappa(\star)}.
\]
Then by definition, $g= \frac{r}{z} \in B_{\pi(\kappa(\star))}$, so that $g \in A_{\sigf(\kappa(\star))}$, by definition of $\sigf$.

Conversely, suppose $\frac{r}{z} = g \in A_{\sigf(\kappa(\star))}$. Let $B \subseteq A$ with $B \in \ffid(R)$ and $g \in B_{\pi(\kappa(\star))}$.  Let $x \in \nzd(R)$ such that $xB \subseteq R$.  Then by definition, \[
xr \in (zxB)^{\kappa(\star)} = (zxB)_\star \cap R \subseteq (zxB)_\star = xz \cdot B_\star.
\]
Then dividing by the unit $zx$ of $Q$, we have $g=\frac{r}{z} \in B_\star$.  Since $B$ is a finitely generated $R$-submodule of $A$, it follows that  $g \in A_{\star_\rf}$, as required.

The last statement is obvious.
\end{proof}

\begin{prop}\label{pr:inv2}
Let $\rc$ be a standard closure operation.  Then $\kappa(\sigf(\rc)) = \rc_\rf$.  In particular, if $\rc$ is of finite type, then $\kappa(\sigf(\rc)) = \rc$.
\end{prop}

\begin{proof}
Let $I \in \id(R)$ and $r \in R$.  

If $r \in I^{\rc_\rf}$, then there is some $J \in \idf(R)$ with $J \subseteq I$ and $r \in J^\rc$, by definition of $\rc_\rf$.  Clearly, then $r \in J_{\pi(\rc)}$ (by using the denominators $1$ and $1$), so that since $J$ is finitely generated, $r \in I_{\sigf(\rc)} \cap R = I^{\kappa(\sigf(\rc))}$.  

If, conversely, $r \in  I^{\kappa(\sigf(\rc))}$, this means that $r \in I_{\sigf(\rc)} \cap R$, so that there is some finitely generated ideal $J \subseteq I$ with $r \in J_{\pi(\rc)}$.  Again using denominators $1$ and $1$, it follows that $r \in J^\rc$, whence $r \in I^{\rc_\rf}$.

The last statement is obvious.
\end{proof}

\begin{proof}[Proof of Theorem~\ref{thm:main}]
All that remains after Theorem~\ref{thm:inv1} and Proposition~\ref{pr:inv2} is order preservation.  So let $\star$ and $\circ$ be semistar operations such that $\star \leq \circ$ -- that is, $A_\star \subseteq A_{\circ}$ for all $A \in \fsub(R)$.  Then in particular this holds for all ideals $I$ of $R$, so we have \[
I^{\kappa(\star)} = I_\star \cap R \subseteq I_{\circ} \cap R = I^{\kappa(\circ)}.
\]
Hence, $\kappa$ is order-preserving.

To see that $\sigf$ is also order-preserving, let $A \in \fsub(R)$ and let $\rc \leq \rc'$ be standard closure operations.  Let $f = r/z \in A_{\sigf(\rc)}$.  Then there is some finitely generated $R$-submodule $B$ of $A$ with $f\in B_{\pi(\rc)}$.  Let $x \in \nzd(R)$ with $xB \subseteq R$.  Then $xr \in (zxB)^\rc \subseteq (zxB)^{\rc'}$, whence $f \in B_{\pi(\rc')}$, so that $f\in A_{\sigf(\rc')}$, as required.
\end{proof}

\begin{defn}
Say that a semistar operation $\star$ is \emph{of fractional type} if for all $A \in \fsub(R)$, one has \[
A_\star = \bigcup \{ B_\star \mid B \subseteq A, B \in \fid(R) \}.
\]
\end{defn}

Note: When $R$ is Noetherian, ``fractional type'' coincides with ``finite type'' for semistar operations.

It is natural to ask whether all semistar operations are of fractional type.  The answer in general, however, is no.  The following is essentially due to Marco Fontana \cite{Fo-pers}:

\begin{prop}
Let $R$ be a commutative ring.  Then all semistar operations on $R$ are of fractional type if and only if $\fsub(R) = \fid(R) \cup \{Q\}$.
\end{prop}

\begin{proof}
It is obvious that if $\fsub(R) = \fid(R) \cup \{Q\}$, any semistar operation must be of fractional type.

Suppose, on the other hand, that $\fsub(R) \neq \fid(R) \cup \{Q\}$.  Define $\rg: \fsub(R) \ra \fsub(R)$ as follows:
\[
A_\rg := \begin{cases}
A, &\text{if } A \in \fid(R),\\
Q, &\text{otherwise.}
\end{cases}
\]
It is easy to see that $\rg$ is a semistar operation -- the key point is that any $R$-submodule of a fractional ideal is a fractional ideal.  However, choose an element $A \in \fsub(R) \setminus (\fid(R) \cup \{Q\})$.  Then  $A \subseteq \bigcup \{ B_\rg \mid B \subseteq A, B \in \fid(R) \} = \bigcup \{B \mid B \subseteq A, B \in \fid(R) \} \subseteq A$, so all are equalities.  But $A_\rg = Q \neq A$.  Thus, $\rg$ is not of fractional type.
\end{proof}

Such rings are rare.  For instance any such ring must clearly be \emph{conducive} (i.e. for any \emph{ring} $T$ with $R \subseteq T \subseteq Q$, $T$ is a fractional $R$-ideal) \cite{DoFe-conducive}, and if $R$ is a Noetherian conducive integral domain, then $R$ is local of dimension $\leq 1$ \cite[Corollary 2.7]{DoFe-conducive}.

\begin{rmk}
Let $\rc$ be a standard preclosure operation on $R$.  Let $\pi(\rc): \fid(R) \ra \fsub(R)$ be as in Definition~\ref{def:sigf}, and define $\sig(\rc): \fsub(R) \ra \fsub(R)$ by \[
A_{\sig(\rc)} := \bigcup \{ B_{\pi(\rc)} \mid B \subseteq A \text{ and } B \in \fid(R)\},
\]
As noted above, if $R$ is Noetherian, then $\fid(R) = \ffid(R)$, so in this case $\sig=\sigf$ and if $\rc$ is a standard closure operation then $\sig(c)=\sigf(c)$ is a semistar operation.

Even in the non-Noetherian case, the proofs of the results in this section show that $\sig(\rc)$ is extensive, order-preserving, and divisible.  However, even if $\rc$ is a \emph{closure} operation, the author does not know whether $\sig(\rc)$ is \emph{idempotent} (the only obstruction to $\sig(\rc)$ being a semistar operation).  If so, perhaps this could provide the basis for a bijection between the poset of standard closure operations on $R$ and the poset of \emph{semistar operations on $R$ of fractional type}.  In other words, we propose the following questions: \begin{enumerate}
\item If $\rc$ is a standard closure operation on $R$, must $\sig(\rc)$ be a semistar operation on $R$? (It clearly has fractional type.)
\item If so, do $\sig$ and $\kappa$ create a bijection between the semistar operations on $R$ of fractional type and the standard closure operations on $R$?
\end{enumerate}
\end{rmk}

\begin{example}
It is worth looking at one example of the correspondence outlined above.

Let $R$ be a commutative ring and $Q$ its total ring of fractions.  We define the \emph{b-operation} $\rb: \fsub(R) \ra \fsub(R)$ as follows.  For $A \in \fsub(R)$ and $n \in \N$, let $A^n$ be the $R$-submodule of $Q$ generated by all $n$-fold products of elements of $A$.  Then for $x\in R$, we say $x \in A_\rb$ if there is some $n \in \N$ and there exist elements $a_1 \in A, a_2 \in A^2, \dotsc, a_n \in A^n$ such that \[
x^n + a_1 x^{n-1} + a_2 x^{n-2} + \cdots + a_n = 0.
\]
In the case where $R$ is an integral domain, this is the usual definition of the b-operation, and the usual proofs go through to show that it is a semistar operation for arbitrary $R$.  Note also that it is of finite type.

On the other hand, one defines the \emph{integral closure} of ideals $-: \id(R) \ra \id(R)$ as follows:
For $r\in R$ and $I \in \id(R)$, we say $r \in I^-$ if there is some $n \in \N$ and there exist elements $i_1 \in I, i_2 \in I^2, \dotsc, i_n \in I^n$ such that \[
r^n + i_1 r^{n-1} + i_2 r^{n-2} + \cdots + i_n = 0.
\]
This is the usual definition of integral closure of ideals (cf. \cite{NR} or \cite{HuSw-book}), and it is  a standard closure operation \cite[Proposition 4.1.3(ii)]{nme-guide2} of finite type.

An easy computation shows that $\kappa(\rb) = -$ and $\sigf(-) = \rb$.
\end{example}

\section{The standard closure operation associated to radical}
Let $R$ be any commutative ring.  Define $\rad: \cI(R) \ra \cI(R)$ by $I^\rad := \{f \in R \mid \exists n \in \N: f^n \in I\}$, called the \emph{radical} (or sometimes the \emph{nilradical}) of an ideal $I$.  This notion is doubtless familiar; only the notation is unusual.  

Radical is clearly a weakly prime closure operation of finite type.  However, it is almost never standard.  To see this, we begin with the following characterization of $\rad_\rs$:

\begin{thm}\label{thm:rads}
Let $R$ be a commutative ring, let $Q$ be its total ring of fractions, and let $I$ be an ideal of $R$.  Then \[
I^{\rad_\rs} = (IQ)^\rad \cap R.
\]
\end{thm}

\begin{proof}
Let $W$ be the set of regular elements of $R$, so that $Q=R_W$.  Let $x \in I^{\rad_\rs}$.  Then for some $w\in W$, we have $wx \in (wI)^\rad$.  That is, there is some $n \in \N$ such that $(wx)^n \in wI$, so that $w^{n-1} x^n \in I$.  It follows that \[
(x/1)^n = \frac{x^n}{1} = \frac{1}{w^{n-1}} \cdot \frac{w^{n-1}x^n}{1} \in \frac{1}{w^{n-1}}IQ = IQ,
\]
since $w^{n-1}/1$ is a unit of $Q$.  Thus $x/1 \in (IQ)^\rad$, whence $x \in (IQ)^\rad \cap R$.

Conversely, let $x \in (IQ)^\rad \cap R$.  Then for some $n \in \N$, we have $x^n/1 \in IQ$.  Without loss of generality, we may choose $n \geq 2$.  But then there is some $w\in W$ such that $wx^n \in I$.  Multiplying by $w^{n-1}$, we have $(wx)^n = w^n x^n \in w^{n-1}I \subseteq wI$, which means that $wx \in (wI)^\rad$, whence $x \in ((wI)^\rad :_R w) \subseteq I^{\rad_\rs}$, as was to be shown.
\end{proof}

We immediately obtain the following:

\begin{cor}\label{cor:radns}
Let $R$ be a commutative ring.  Then $\rad_\rs = \rad$ if and only if all non-zerodivisors of $R$ are units.  Indeed, for any ideal $I$ that contains a non-zerodivisor, $I^{\rad_\rs} = R$.
\end{cor}

Next we show what happens with primary ideals and primary decomposition.  First, note the following fact about intersections:

\begin{cor}\label{cor:radscap}
Let $R$ be a commutative ring and $I_1, \dotsc, I_n$ ideals of $R$.  Then $(I_1 \cap \cdots \cap I_n)^{\rad_\rs} = I_1^{\rad_\rs} \cap \cdots \cap  I_n^{\rad_\rs}$.
\end{cor}

\begin{proof}
Let $Q$ be the total quotient ring of $R$.  We have \begin{align*}
\bigcap_{j=1}^n I_j^{\rad_\rs} &= \bigcap_{j=1}^n \left((I_j Q)^\rad \cap R\right) = \Big(\bigcap_{j=1}^n (I_j Q)^\rad\Big) \cap R  \\
&= \Big(\bigcap_{j=1}^n (I_j Q)^\rad\Big) \cap R = \Big(\big(\bigcap_{j=1}^n I_j\big)Q\Big)^\rad \cap R = \left(\cap_{j=1}^n I_j\right)^{\rad_\rs}. \qedhere
\end{align*}
\end{proof}

\begin{prop}\label{pr:radsprimary}
Let $R$ be a commutative ring, and let $I$ be a \emph{primary} ideal.  Then \[
I^{\rad_\rs} = \begin{cases}
I^\rad &\text{if all elements of $I$ are zero-divisors of $R$,}\\
R &\text{otherwise.} 
\end{cases}
\]
\end{prop}

\begin{proof}
If $I$ contains a non-zerodivisor, Corollary~\ref{cor:radns} shows $I^{\rad_\rs} = R$.

So assume that all elements of $I$ are zero-divisors. Let $a\in ((wI)^\rad : w)$, where $w$ is a non-zerodivisor, so that $wa \in (wI)^\rad$.  Then there is some $n$ with $w^n a^n \in wI$, so that $w^{n-1} a^n \in I$.  But since $w^{n-1} \notin I$ and $I$ is primary, it follows that some power of $a^n$ is in $I$, whence $a \in I^\rad$.  Thus, \[
I^\rad \subseteq I^{\rad_\rs} = \bigcup_{w\in \nzd(R)} ((wI)^\rad :w) \subseteq I^\rad,
\]
so that all are equal.
\end{proof}

The following then gives a way to compute the standardized radical of any ideal of a Noetherian ring, without changing rings to the total ring of fractions.
\begin{thm}\label{thm:radspdec}
Let $R$ be a commutative ring. Let $I$ be an ideal that has a primary decomposition.  Represent $I$ as follows: \[
I = \q_1 \cap \cdots \cap \q_n \cap \q_{n+1} \cap \cdots \cap \q_k,
\]
where each of $\q_1, \ldots \q_n$ are primary ideals consisting of zero-divisors, and each $\q_i$, $i>n$, contains a regular element.  Say each $\q_i$ is $\p_i$-primary, so that $\p_i = \q_i^\rad$.  Then \[
I^{\rad_\rs} = \p_1 \cap \cdots \cap \p_n.
\]
\end{thm}

\begin{proof}
By Proposition~\ref{pr:radsprimary} and Corollary~\ref{cor:radscap}, we have \begin{align*}
I^{\rad_\rs} &= \left(\bigcap_{i=1}^k \q_i\right)^{\rad_\rs} = \bigcap_{i=1}^k \q_i^{\rad_\rs} = \left(\bigcap_{i=1}^n \q_i^{\rad_\rs}\right) \cap \left(\bigcap_{i=n+1}^k \q_i^{\rad_\rs}\right) \\
&=\left(\bigcap_{i=1}^n \q_i^\rad\right) \cap \left(\bigcap_{i=n+1}^k R\right) = \bigcap_{i=1}^n \p_i. \qedhere
\end{align*}
\end{proof}

\begin{example}
It can happen that a radical ideal consisting of zero-divisors has a nontrivial standardized radical.  That is, the assumption of primariness is necessary in Proposition~\ref{pr:radsprimary}.

Let $k$ be a field and let $R = k[X,Y,Z]/(X^2, XY) = k[x,y,z]$, where $X,Y,Z$ are indeterminates over $k$ and $x,y,z$ are their homomorphic images in $R$.  Let $I := (x, yz)$.  Then $I$ is a radical ideal, since $R/I \cong k[Y,Z]/(YZ)$ is a reduced ring.  Moreover, $I$ consists of zero-divisors, as $xI=0$.  However, $I^{\rad_\rs} = (x,y)$, as follows from Theorem~\ref{thm:radspdec}, since we have the primary decomposition $I = (x,y) \cap (x,z)$, but $(x,z)$ contains $z$, an $R$-regular element.
\end{example}

\section*{Acknowledgment}
The author would like to extend his thanks to the referee for useful comments that improved the paper.

\providecommand{\bysame}{\leavevmode\hbox to3em{\hrulefill}\thinspace}
\providecommand{\MR}{\relax\ifhmode\unskip\space\fi MR }
\providecommand{\MRhref}[2]{%
  \href{http://www.ams.org/mathscinet-getitem?mr=#1}{#2}
}
\providecommand{\href}[2]{#2}

\end{document}